\documentclass[11pt,twoside]{amsart}
\usepackage{latexsym,amssymb,amsmath}
\usepackage{tikz}
\usepackage{float}
\usepackage{amsmath,amsfonts,amsthm,amssymb} 
\usepackage{latexsym,amscd,verbatim,alltt,array}
\usepackage{enumerate}
\usepackage{tikz}
\usetikzlibrary{arrows}
\usetikzlibrary{shapes, shadows, arrows}

\numberwithin{equation}{section}
\newtheorem{theorem}{Theorem}[section]
\newtheorem{lemma}[theorem]{Lemma}
\newtheorem{proposition}[theorem]{Proposition}
\newtheorem{corollary}[theorem]{Corollary}

\theoremstyle{definition}
\newtheorem{definition}[theorem]{Definition} 
 
\newtheorem{remark}[theorem]{Remark}
\newtheorem{example}[theorem]{Example}

\begin{document}


\title[Unmixed and Cohen--Macaulay weighted oriented graphs]{Unmixed
and Cohen--Macaulay weighted oriented K\"onig graphs}  

\thanks{The first author was
supported by a scholarship from CONACYT, Mexico. The second and third
author were supported by SNI, Mexico.}

\author{Yuriko Pitones}
\address{
Departamento de
Matem\'aticas\\
Centro de Investigaci\'on y de Estudios
Avanzados del
IPN\\
Apartado Postal
14--740 \\
07000 Mexico City, D.F.
}
\email{ypitones@math.cinvestav.mx}

\author{Enrique Reyes}
\address{
Departamento de
Matem\'aticas\\
Centro de Investigaci\'on y de Estudios
Avanzados del
IPN\\
Apartado Postal
14--740 \\
07000 Mexico City, D.F.
}
\email{ereyes@math.cinvestav.mx}

\author{Rafael H. Villarreal}
\address{
Departamento de
Matem\'aticas\\
Centro de Investigaci\'on y de Estudios
Avanzados del
IPN\\
Apartado Postal
14--740 \\
07000 Mexico City, D.F.
}
\email{vila@math.cinvestav.mx}

\begin{abstract} 
Let $D$ be a weighted oriented graph, whose underlying graph is $G$, and
let $I(D)$ be its edge ideal. If $G$ has no 
$3$-, $5$-, or $7$-cycles, or $G$ is  K\"{o}nig, we characterize when
$I(D)$ is unmixed. If $G$ has no $3$- or
$5$-cycles, or $G$ is
K\"onig, we
characterize when $I(D)$ is Cohen--Macaulay. We prove that
$I(D)$ is unmixed if and only if $I(D)$ is Cohen--Macaulay when 
$G$ has girth greater than $7$ or $G$ is K\"onig and
has no $4$-cycles.    
\end{abstract}

\maketitle 

\section{Introduction}\label{intro-section}
Let $G=(V(G), E(G))$ be a simple graph without isolated vertices 
with vertex set $V(G)$ and edge
set $E(G)$. A {\it weighted oriented graph\/} $D$, whose {\it underlying
graph\/} is $G$, is a triplet $(V(D),E(D),w)$ where $V(D)=V(G)$,
$E(D)\subset V(D)\times V(D)$ such that 
$$
E(G)=\{\{x,y\}\mid (x,y)\in
E(D)\},
$$ 
$|E(D)|=|E(G)|$, and $w$ is a \textit{weight function} 
$w\colon V(D) \to\mathbb{N}_+$. Here $\mathbb{N}_+$ denotes the set 
of positive integers. The \textit{vertex set} of $D$ and the \textit{edge set} of $D$
are $V(D)$ and $E(D)$, respectively. Sometimes for simplicity we denote
these sets by $V$ and $E$, respectively. The \emph{weight} of $x\in V$
is $w(x)$ and the set of vertices $\{x\in V\mid w(x)>1\}$ is denoted by
$V^{+}$. If $V(D)=\{x_1,\ldots,x_n\}$, we can regard each vertex $x_i$ as a
 variable and consider the polynomial ring
 $R=K[x_1,\ldots,x_n]$ over a ground field $K$. The
 \textit{edge ideal} of $D$, introduced in \cite{vivares,WOG}, 
is the ideal of $R$ given by 
$$I(D):=(x_{i}x_{j}^{w(x_{j})}:(x_{i},x_{j})\in E(D)).$$
\quad If $w(x)=1$ for each $x\in V(D)$, then $I(D)$ is the usual edge
ideal $I(G)$ of the graph $G$ \cite{cm-graphs}, which has been
extensively studied in the literature (see
\cite{graphs-rings,monalg-2} and the references therein). The
motivation to study $I(D)$ comes from coding theory, see
\cite[p.~536]{reyes-vila} and \cite[p.~1]{WOG}. 

In general, edge
ideals of weighted oriented 
graphs are different from edge ideals of 
edge-weighted (undirected) graphs defined by Paulsen and
Sather-Wagstaff \cite{PS}.  Consider the weighted
oriented graph $D' = (V,E,w')$ with $w'(x) = 1$ if $x$ is a
\textit{source} vertex (i.e., a vertex
with only outgoing edges) and $w'(x) =
w(x)$ if $x$ is not a source vertex. 
Then, $I(D') = I(D)$. In view of this throughout this paper, we will
always assume that if $x$ is a source, then $w(x) = 1$. 

The projective
dimension, regularity, and algebraic and
combinatorial properties of edge ideals of
weighted oriented graphs have been studied in
\cite{vivares,reyes-vila,depth-monomial,
WOG,Zhu-Xu,Zhu-Xu-Wang-Zhang}. The first major
result about $I(D)$ is an explicit combinatorial
expression of Pitones, Reyes and Toledo \cite[Theorem~25]{WOG} 
for the irredundant decomposition of $I(D)$ as a finite intersection of
irreducible monomial ideals. If $D$ is transitive, then Alexander duality holds
for $I(D)$ \cite[Theorem~4]{vivares}.

The edge ideal $I(D)$ is \textit{unmixed} if all its associated primes have the same height and
$I(D)$ is \textit{Cohen--Macaulay} if $R/I(D)$ is a Cohen--Macaulay
ring in the sense of \cite{Mats}. We say that $D$ is \textit{unmixed} (resp.
\textit{Cohen--Macaulay}) if $I(D)$ is unmixed (resp. Cohen--Macaulay). As pointed out in
\cite[p.~536]{reyes-vila}, the Cohen--Macaulay property and the unmixed property of $I(D)$ are
independent of the weight we assign to a \emph{sink} vertex (i.e.,
a vertex with only incoming edges). For this reason, we shall always
assume---when studying these properties---that sinks have weight $1$.

The graph $G$ is \textit{well-covered} if all maximal stable sets of
$G$ have the same cardinality and the graph $G$ is  \textit{very
well-covered} if $G$ is well-covered and $|V(G)|=2\tau(G)$, where
$\tau(G)$ is the cardinality of a minimum vertex cover of $G$.  
The class of very well-covered graphs contains in
particular the bipartite well-covered graphs studied by Ravindra
\cite{ravindra} and more recently revisited in \cite{unmixed}. One of
the properties of very well-covered graphs that will be used
in this paper is that they can be classified using combinatorial
properties of a perfect matching as shown by a central result of
Favaron \cite[Theorem~1.2]{favaron} (Theorem~\ref{konig}). 

The content of this paper is as follows. In
Section~\ref{prelim-section} we  
present some well-known
results about edge ideals. We 
denote the in- and out-neighborhood of a vertex $a$ by 
$N_{D}^{-}(a)$ and $N_{D}^{+}(a)$, respectively, and the neighborhood 
of $a$ by $N_{D}(a)$ (Definition~\ref{in-out-neighborhood}).

Let $D=(V(D),E(D),w)$ be a weighted oriented graph without isolated vertices whose
underlying graph is $G$. The graph $G$ is \textit{K\"onig} if
$\tau(G)$ is the \textit{matching number} of $G$, that is, the maximum
cardinality of a matching of $G$.  
In Section~\ref{unmixed-section}, we characterize in graph theoretical terms the unmixed 
property of $I(D)$ when $G$ is K\"onig. 

One of our main results is: 

\noindent {\bf Theorem~\ref{prop-unmixed}.}\textit{ 
If $G$ is K\"onig, then $I(D)$ is unmixed if and only if 
the following conditions hold: 
\begin{enumerate}[{\rm (1)}]
\item $G$ has a perfect matching $P$ with property {\bf(P)}, that is,
$G$ has a perfect matching $P$ such that if $\{a,b\}$, $\{a^{\prime},b^{\prime}\}\in E(G)$, and
$\{b,b^{\prime}\}\in P$,  
then $\{a,a^{\prime}\}\in E(G)$.  
\item If $a$ is a vertex of $D$ with $w(a)>1$, $b^{\prime}\in
N_{D}^{+}(a)$, and $\{b,b^{\prime}\}\in P$, then $N_{D}(b)\subset
N_{D}^{+}(a)$.
\end{enumerate}}
We also show that conditions (1) and (2) of Theorem~\ref{prop-unmixed}
characterize the unmixed property of $I(D)$ when 
$G$ is a graph without $3$-, $5$-, and $7$-cycles
(Proposition~\ref{Coro-37}). We give other characterizations of the
unmixed property of $I(D)$ when $G$ is a K\"onig graph
(Corollary~\ref{jun24-19}) or $G$ is very well-covered (Corollary~\ref{sep15-19}). 

The Cohen--Macaulay property of $I(D)$ is studied in
Section~\ref{c-m-section}. We give 
a combinatorial characterization of this property when $G$ is K\"onig.
 
Another of our main results is: 

\noindent {\bf Theorem~\ref{theoremCM}.}\textit{
If $G$ is K\"{o}nig, then $I(D)$ is Cohen--Macaulay if and only if $D$
satisfies the following two conditions:  
\begin{enumerate}[{\rm (1)}]
\item $G$ has a perfect matching $P$ with property {\bf(P)} and $G$ has no $4$-cycles 
with two edges in $P$.
\item If $a\in V(D)$, $w(a)>1$, $b^{\prime}\in N_{D}^{+}(a)$, 
and $\{b,b^{\prime}\}\in P$, then $N_{D}(b)\subset N_{D}^{+}(a)$.
\end{enumerate}}

We also show that conditions (1) and (2) of Theorem~\ref{theoremCM}
characterize the Cohen--Macaulay property of $I(D)$ when 
$G$ is a graph without $3$- and $5$-cycles (Proposition~\ref{Coro-45}).
In general any graded Cohen--Macaulay ideal is unmixed \cite{Mats}. If
$G$ is a K\"onig graph without $4$-cycles or
$G$ has girth greater than $7$, we prove 
that $I(D)$ is unmixed if and only if $I(D)$ is 
Cohen--Macaulay (Corollaries~\ref{jun23-19} and
\ref{jun25-19}). For graphs this improves a result of
\cite[Corollary~2.19]{susan-reyes-vila} showing that unmixed K\"onig
clutters without $3$- and $4$-cycles are
Cohen--Macaulay. If $I(D)$ is Cohen--Macaulay, then $I(D)$ is unmixed and $I(G)$ is
Cohen--Macaulay (see \cite[Theorem~2.6]{Radical-Herzog} and 
\cite[Proposition 51]{WOG}). 
The converse is a conjecture 
\cite[Conjecture~53]{WOG}. We prove this conjecture, when $G$ has no
$3$- or $5$- cycles, or $G$ is K\"onig (Corollary~\ref{conjecture-WOG}).

Graphs with a whisker (i.e., pendant edge) attached to 
each vertex are K\"onig \cite[p.~277]{monalg-2},  
very well-covered graphs are also K\"onig (Remark~\ref{2star}), and bipartite
graphs are K\"onig and have no odd cycles \cite{Har}. Then some of our
results generalize those of \cite{vivares,reyes-vila,WOG,unmixed}. More
precisely, Theorem~\ref{prop-unmixed} (resp.
Theorem~\ref{theoremCM}) generalizes the unmixed criteria of
\cite[Theorem~46]{WOG} and \cite[Theorem~1.1]{unmixed} (resp.
Cohen--Macaulay criterion of \cite[Theorem~5.1]{reyes-vila}) 
 for weighted oriented bipartite graphs. 
From Corollary~\ref{jun23-19} and Proposition~\ref{Coro-45},
 we recover the Cohen--Macaulay criterion of \cite[Theorem~5]{vivares}
 for weighted oriented trees. Finally in Section 5, we show some
 weighted oriented graphs that exemplify our results.  

For all unexplained terminology and additional information,  we refer to 
\cite{digraphs} for the theory of digraphs, 
and \cite{graphs-rings,monalg-2} for the theory of
edge ideals and monomial ideals.

\section{Preliminaries}\label{prelim-section}
In this section we give some definitions and present some well-known
results that will be used in the following sections.  
To avoid repetitions, we continue to employ 
the notations and
definitions used in Section~\ref{intro-section}.

\begin{definition}\label{in-out-neighborhood}
Let $x$ be a vertex of a weighted oriented graph $D$. The sets
$$
N_{D}^{+}(x):=\{y\in V(D)\mid(x,y)\in E(D)\}\ \mbox{ and }\ 
N_{D}^{-}(x):=\{y\in\ V(D)\mid(y,x)\in E(D)\}
$$
are called the {\it out-neighborhood\/} and the 
{\it in-neighborhood\/} of $x$, respectively. The {\it
neighborhood\/} of $x$ is the set 
$N_{D}(x):=N_{D}^{+}(x)\cup N_{D}^{-}(x)$. If $A\subset V(D)$, we 
set $N_D^+(A):=\bigcup_{a\in A} N_D^+(a)$. 
\end{definition}

\begin{definition}
A {\it vertex cover\/} $C$ of a weighted oriented graph $D$ is a subset of $V(D)$, such that if
$(x,y)\in E(D)$, then $x\in C$ or $y\in C$. A vertex cover $C$ of $D$
is {\it minimal\/} if each proper subset of $C$ is not a vertex cover
of $D$.   
\end{definition}

If $G$ is the underlying graph of $D$, then $C$ is a (minimal) vertex
cover of $G$ if and only if $C$ is a (minimal) vertex cover of $D$. 

\begin{definition}\label{L-sets}
Let $C$ be a vertex cover of a weighted oriented graph $D$, we define
the following three sets that form a partition of $C$: 
\begin{itemize}
\item $L_1(C):=\{x\in C\mid N_{D}^{+}(x)\cap C^{c}\neq \emptyset \}$,
 where $C^{c}:=V(D)\setminus C$, 
\item $L_2(C):=\{x\in C\mid\mbox{$x\notin L_1(C)$ and
$N^{-}_{D}(x)\cap C^c\neq\emptyset$}\}$, and 
\item$L_3(C):=C\setminus(L_1(C)\cup L_2(C))$.
\end{itemize}
\end{definition}

\begin{lemma}\cite[Proposition~5]{WOG}\label{Rem-L3}
Let $C$ be a vertex cover of a weighted oriented graph $D$ and let $x$
be a vertex in $C$,  
then $x\in L_3(C)$ if and only if $N_{D}(x)\subset C$.
\end{lemma}

\begin{definition}\label{strong-cover-defn}
A vertex cover $C$ of a weighted oriented graph $D$ is {\it strong\/} if for each $x\in L_3(C)$
there is $(y,x)\in E(D)$ such that $y\in L_2(C)\cup L_{3}(C)$ and 
$y\in V^{+}$ (i.e., $w(y)>1$).
\end{definition}

\begin{theorem}\label{theorem42}{\rm \cite[Theorem~31]{WOG}}
Let $D$ be a weighted oriented graph and let $G$ be its underlying
graph. The following conditions are equivalent: 
\begin{enumerate}[{\rm (1)}]
\item $I(D)$ is unmixed.
\item All strong vertex covers of $D$ have the same cardinality. 
\item $I(G)$ is unmixed and $L_{3}(C)=\emptyset$ for each strong vertex cover $C$ of $D$.
\end{enumerate}
\end{theorem}

\begin{proposition}\label{teorema-estrella}{\rm \cite[Proposition 51]{WOG}}
Let $D$ be a weighted oriented graph and let $G$ be its underlying graph. 
If $I(D)$ is Cohen--Macaulay, then all strong vertex covers of $D$ are
minimal vertex covers and $I(G)$ is Cohen--Macaulay.
\end{proposition}

\begin{proposition}\label{propo-estrella}{\rm \cite[Corollary 6]{vivares}}
Let $D=(V,E,w)$ be a weighted oriented graph and let 
$D^{\prime}=(V,E,w^{\prime})$ be the weighted oriented graph with
$w^{\prime}(x)=2$ if $w(x)\geq 2$
and $w^{\prime}(x)=1$ if $w(x)=1$. 
Then, $I(D)$ is Cohen--Macaulay if and only if $I(D^{\prime})$ is Cohen--Macaulay.  
\end{proposition}

\begin{definition}
The {\it cover number\/} of a graph $G$, denoted by $\tau(G)$, is the cardinality of a minimum vertex cover of $G$.
\end{definition}

\begin{definition}\label{konig-def}
A collection of pairwise disjoint edges of a graph $G$ is called a
{\it matching\/}. A matching $P=\{e_1,\ldots,e_g\}$ is \textit{perfect} if
$V(G)=\bigcup_{i=1}^ge_i$.
$G$ is a {\it K\"{o}nig graph\/} if $\tau(G) =\nu(G)$ where $\nu(G)$
is the \textit{matching number} of $G$, that is, the maximum
cardinality of a matching of $G$.   
\end{definition} 

\begin{definition}
Let $G$ be a graph, a {\it stable set\/} of $G$ 
is a subset of $V(G)$ containing no edge of $G$. The graph $G$ is
{\it well-covered\/} 
if all maximal stable sets of $G$ have the same cardinality.
\end{definition}

\begin{remark}\rm\label{1star} Let $G$ be a graph. 
A set of vertices $S$ is a (maximal) stable set of $G$ if and only if
$V(G)\setminus 
S$ is a (minimal) vertex cover of $G$. The edge ideal $I(G)$ is unmixed if and
only if all minimal vertex covers of $G$ have the same cardinality.  
Then, the edge ideal $I(G)$ is unmixed if and only if $G$ is well-covered.
\end{remark}

\begin{definition}
A graph $G$ is called \textit{very well-covered} if it is
well-covered, has no isolated vertices, and $|V(G)|= 2\tau(G)$.
\end{definition}

\begin{definition}\label{p-property-def}
Let $P$ be a perfect matching of a graph $G$. If $\{a,a^{\prime}\}$ is
an edge of $G$ 
for all $\{a,b\}, \{a^{\prime},b^{\prime}\}\in E(G)$ and
$\{b,b^{\prime}\}\in P$, then we say that $P$ {\it satisfies property {\bf(P)}\/}. 
\end{definition}

\begin{remark}\label{sep26-19}
Let $P$ be a perfect matching of a graph $G$ with property {\bf(P)}.
Note that if $\{b,b'\}$ is in $P$ and $a\in V(G)$, then
$\{a,b\}$ and $\{a,b'\}$ cannot be both in $E(G)$ because $G$ 
has no loops. 
\end{remark}

A bipartite graph $G$ without isolated vertices is
unmixed if and only if $G$ has a perfect matching $P$ that satisfies 
property {\bf(P)} \cite[Theorem~1.1]{unmixed}. The next result
generalizes this fact.

\begin{theorem}\label{konig}{\rm \cite[Theorem 1.2]{favaron}}
The following conditions are equivalent for a graph $G$: 
\begin{enumerate}[{\rm (a)}]
\item $G$ is very well-covered.
\item $G$ has a perfect matching with the property {\bf{(P)}}. 
\item $G$ has at least one perfect matching, and each perfect
matching of $G$ satisfies {\bf{(P)}}.
\end{enumerate}
\end{theorem}

\begin{remark}\rm\label{2star}
If $G$ is a very well-covered graph, then $|V(G)|= 2\tau(G)$. Furthermore, by Theorem \ref{konig}, $G$ has a perfect matching, then $2\nu(G)=|V(G)|$. Therefore $G$ is a K\"{o}nig graph.
\end{remark}

\begin{theorem}\label{lemma-Ivan}{\rm
(\cite[Theorem 5]{disc-math}, \cite[Lemma~2.3]{susan-reyes-vila})} 
Let $G$ be a graph without isolated vertices. If $G$ is a graph
without $3$-, $5$-, and $7$-cycles or $G$ is a K\"{o}nig graph, then $G$ is
well-covered if and only if $G$ is very well-covered. 
\end{theorem}

\begin{theorem}\label{3-5cycles} 
If $G$ is a graph without $3$- and $5$-cycles or $G$ is a K\"onig graph,
then the following conditions are equivalent: 
\begin{enumerate}[{\rm (a)}]
\item $I(G)$ is Cohen--Macaulay.
\item If $H$ is a connected component of $G$, then $H$ is an isolated
vertex or $H$ has a perfect matching $P$ with the property {\bf{(P)}}
and there are no $4$-cycles of $H$ with two edge in $P$. 
\end{enumerate}
\end{theorem}

\begin{proof} (a)$\Rightarrow$(b): Let $H$ be a connected component of
$G$ which is not an isolated vertex. First assume that $G$ has no 
$3$- or $5$-cycles .
According to \cite[Lemma~4.1]{cm-graphs}, $H$ is Cohen--Macaulay. Then,
by \cite[Theorem~32(d)]{Ivan-Reyes}, $H$ is unmixed, has a perfect
matching $P=\{e_1,\ldots,e_g\}$ with $g=|P|=\tau(G)$, and has no 
$4$-cycles containing two $e_i$'s. Then, $H$ is very well-covered
since $|V(G)|=2\tau(G)$. Therefore, by Theorem~\ref{konig}(c), $P$
has property {\bf{(P)}}. Now assume  $G$ is K\"onig. As $H$ is Cohen--Macaulay
\cite[Lemma~4.1]{cm-graphs}, by 
\cite[Proposition~28(iv)]{Ivan-Reyes}, $H$ is well-covered, has a perfect
matching $P$ with $|P|=\tau(G)$, and has no $4$-cycles with two edges
in $P$. Then, $H$ is very well-covered 
since $|V(G)|=2\tau(G)$. Therefore, by Theorem~\ref{konig}(c), $P$
has property {\bf{(P)}}.

(b)$\Rightarrow$(a): 
Let $H$ be a connected component of $G$ which is 
not an isolated vertex. First assume that $G$ has no 
$3$- or $5$-cycles.  The graph $G$ is Cohen--Macaulay if and only
if all connected components of $G$ are Cohen--Macaulay
\cite[Lemma~4.1]{cm-graphs}. Thus we need only show that $H$ is
Cohen--Macaulay. As $H$ has a perfect matching $P$ with property
{\bf{(P)}}, by Theorem~\ref{konig} and Remark~\ref{2star}, $H$ is
very well-covered and K\"onig. Hence,
$H$ satisfies the hypothesis of \cite[Theorem~32(d)]{Ivan-Reyes},
and consequently $H$ is Cohen--Macaulay. Now assume that $G$ is
K\"onig. As before, we need only show that $H$ is
Cohen--Macaulay. As $H$ has a perfect matching $P$ with property
{\bf{(P)}}, by Theorem~\ref{konig}, $H$ is
very well-covered. Hence,
$H$ satisfies the hypothesis of \cite[Theorem~28(iv)]{Ivan-Reyes},
and consequently $H$ is Cohen--Macaulay.
\end{proof}

The next lemma was shown in \cite[Theorem~2.4]{EV} for Cohen--Macaulay
bipartite graphs and was later generalized to Cohen--Macaulay K\"onig graphs. 

\begin{lemma}\label{corolario-grado1}{\rm \cite[Corollary 29]{Ivan-Reyes}}
If $G$ is a Cohen--Macaulay K\"onig graph 
without isolated vertices, then $G$ has a vertex of degree $1$.
\end{lemma}

\section{Unmixed weighted oriented graphs}\label{unmixed-section}
In this section we classify the unmixed property of a weighted
oriented graph $D$ whose underlying graph $G$ is K\"onig. 
Furthermore, we characterize when $I(D)$ is unmixed if $G$ is very
well-covered or $G$ is a graph without $3$-, $5$- and $7$-cycles.

\begin{lemma}\label{lemma2}
Let $G$ be a very well-covered graph with a perfect matching $P$. If
$\{a,b_1\}$, $\{a,b_2\}\in E(G)$ and $\{b_1,b_{1}^{\prime}\},
\{b_2,b_{2}^{\prime}\}\in P$, then $\{b_{1}^{\prime},b_{2}^{\prime}\}\notin E(G)$.  
\end{lemma}
\begin{proof}
By contradiction, suppose $\{b_{1}^{\prime},b_{2}^{\prime}\}\in
E(G)$. Thus, by (c) in Theorem  \ref{konig}, $\{a,b_{2}^{\prime}\}$ is
in $E(G)$, since $\{a,b_{1}\}$ is in $E(G)$ and
$\{b_1,b_{1}^{\prime}\}$ is in $P$.
A contradiction by Remark~\ref{sep26-19}, since $\{a,b_2\}$ is in
$E(G)$ by hypothesis.     
\end{proof}

\begin{lemma}\label{lemma3}
Let $G$ be a very well-covered graph with a perfect matching $P$. If $\{a,b\}\in E(G)$ and $\{b,b^{\prime}\}\in P$, then $N_{G}(b^{\prime})\subset N_{G}(a)$. 
\end{lemma}
\begin{proof}
By (c) in Theorem \ref{konig}, if $c\in N_{G}(b^{\prime})$, then $\{c,a\}\in E(G)$, since $\{b,b^{\prime}\}\in P$. Thus $c\in N_{G}(a)$. Therefore, $N_{G}(b^{\prime})\subset N_{G}(a)$.
\end{proof}

\begin{lemma}\label{remark-(ii)} Let $D$ be a weighted oriented graph
with underlying graph $G$. Suppose $G$ has a perfect matching $P$ 
that satisfies property {\bf{(P)}} such that 
$N_{D}(b)\subset N_{D}^{+}(a)$ for 
$a\in V^{+}$, $b^{\prime}\in N_{D}^{+}(a)$ and $\{b,b^{\prime}\}\in
P$. If $\{c,c^{\prime}\}\in P$ and $c^{\prime}\in N_{D}^{+}(V^+)$,
then $N_{D}^{-}(c)\cap V^{+}=\emptyset$. 
\end{lemma}
\begin{proof}
Since $c^{\prime}\in N_{D}^{+}(V^{+})$, there is $x\in V^{+}$ such
that $c^{\prime}\in N_{D}^{+}(x)$. Then, as $\{c,c'\}\in P$, by
hypothesis $N_{D}(c)\subset N_{D}^{+}(x)$. We take $z\in
N_{D}^{-}(c)$, then $z\in N_{D}(c)\subset N_{D}^{+}(x)$. This
implies that $(x,z)\in E(D)$. Now, if $w(z)>1$, then, as $c\in N_D^+(z)$ and
$\{c,c'\}\in P$, by hypothesis
$N_{D}(c^{\prime})\subset N_{D}^{+}(z)$. But $x\in
N_{D}(c^{\prime})$, then $x\in N_{D}^{+}(z)$, i.e., $(z,x)\in E(D)$.
This is a contradiction, since $(x,z)\in E(D)$. Consequently,
$w(z)=1$. Therefore $N_{D}^{-}(c)\cap V^{+}=\emptyset$.       
\end{proof}
Since the unmixed property of $D$ is closed under connected
components, and isolated vertices are unmixed, in the rest of this
section we assume $D$ does not contains isolated vertices.  

We come to the main result of this section.

\begin{theorem}\label{prop-unmixed} 
Let $D=(V(D),E(D),w)$ be a weighted oriented graph whose underlying graph $G$ is
K\"{o}nig. Then, $I(D)$ is unmixed if and only if $D$ satisfies the following two conditions: 
\begin{enumerate}[{\rm (1)}]
\item $G$ has a perfect matching $P$ with property {\bf(P)}, that is,
$G$ has a perfect matching $P$ such that if 
$\{a,b\}$, $\{a^{\prime},b^{\prime}\}\in E(G)$, and
$\{b,b^{\prime}\}\in P$,  
then $\{a,a^{\prime}\}\in E(G)$.  
\item If $a$ is a vertex of $D$ with $w(a)>1$, $b^{\prime}\in
N_{D}^{+}(a)$, and $\{b,b^{\prime}\}\in P$, then $N_{D}(b)\subset
N_{D}^{+}(a)$.
 \end{enumerate}
\end{theorem}

\begin{proof}
$\Rightarrow)$ By Theorem \ref{theorem42}, $I(G)$ is unmixed. Hence by Remark
\ref{1star} and Theorem \ref{lemma-Ivan}, $G$ is very well-covered,
since $G$ is K\"{o}nig. Thus, by Theorem \ref{konig}, $G$ has a
perfect matching $P$ satisfying condition (1). 
Now, assume $a\in V^{+}$, $b^{\prime}\in N_{D}^{+}(a)$, and
$\{b,b^{\prime}\}\in P$. We set      
\begin{center}
$B:=\{d\in V(G)\mid\textit{there is } 
d^{\prime}\in N_{D}^{+}(a) \textit{ such that } \{d,d^{\prime}\}\in P\}$.
\end{center}
\quad Note that $b\in B$. By Lemma \ref{lemma2}, $B$ is a stable set
of $G$. Pick a maximal
stable set $S$ of $G$ such that $B\subset S$. By Remark \ref{1star}, 
$S^{\prime}:=V(G)\setminus S$ is a
minimal vertex 
cover of $G$. Since $G$ is very-well covered, one has
$$|S|=\tau(G)=\frac{|V(G)|}{2}=|P|.$$
\quad Hence $|S\cap e|=1$ for each
$e\in P$.  Consequently,
$b^{\prime}\in S^{\prime}$, 
since $b\in B\subset S$ and $\{b,b^{\prime}\}\in P$. One has the 
inclusion $N_{D}^{+}(a)\subset S^{\prime}$. Indeed, take $d'\in
N_D^+(a)$. Since $P$ is a perfect matching of $G$, there is $d\in
V(G)$ such that $\{d,d'\}\in P$. As $d\in B$, we get $d\in S$. Hence
$d'\notin S$ because $S$ is a stable set.

We will prove (2) by contradiction. Suppose there is $c^{\prime}\in 
N_{D}(b)\setminus N_{D}^{+}(a)$.  Then, $c^{\prime}\in S^{\prime}$,
since $b\in S$, $\{c',b\}\in E(G)$ and $S$ is stable. We set 
\begin{eqnarray*}
& &A:=\{x^{\prime}\in V(G)\mid \textit{there is } x\in
N_{D}(c^{\prime}) \textit{ such that } \{x,x^{\prime}\}\in
P\},\mbox{ and }\\
& &C:=N_{D}^{+}(a)\cup N_{D}(c^{\prime})\cup (S^{\prime}\setminus A).
\end{eqnarray*}
\quad By
Lemma~\ref{lemma2}, $A$ is a stable set of $G$; and $b'\in A$. As $P$
is a perfect matching of $G$, there is $c\in V(G)$ such
that $\{c,c^{\prime}\}\in P$. Then $c^{\prime}\in A$. Thus,
$c^{\prime}\notin C$, since $c^{\prime}\in S^{\prime}$ and
$c^{\prime}\notin N_{D}^{+}(a)$. Now we take $e\in E(G)$, then there
is $y\in e\cap S^{\prime}$, since $S^{\prime}$ is a vertex cover. If
$y\in S^{\prime}\setminus A$, then $y\in C$. Now, if $y\notin
S^{\prime}\setminus A$, then $y\in S^{\prime}\cap A$ and there is
$y^{\prime}\in N_{D}(c^{\prime})$ such that $\{y,y^{\prime}\}\in P$.
Then, by Lemma \ref{lemma3}, $N_D(y)\subset N_D(c')$. 
Hence, if $e=\{y,y_1\}$, then $y_1\in
N_{D}(c^{\prime})\subset C$. Therefore $C$ is a vertex cover of $G$.

Next we show the inclusion $L_3(C)\subset N_{D}^{+}(a)$. 
Take $x'\in L_3(C)$, i.e., $x'\in C$ and $N_D(x')\subset C$ (see
Lemma~\ref{Rem-L3}). 
If $x'\in N_D(c')$, then
$c'\in N_D(x')\subset C$ and $c'\in C$, which is impossible since
$c^{\prime}\notin C$. Now, assume $x^{\prime}\in S^{\prime}\setminus A$, then
 there is $\{x^{\prime}, x\}\in P$ with $x\in S$, because $P$ is a
 perfect matching of $G$ and $|S'\cap e|=1$ for $e\in P$. 
Consequently,
 $x\notin N_{D}(c^{\prime})$, since $x^{\prime}\notin A$. Then,
 $x\notin C$, since $x\in S$ and $N_{D}^{+}(a)\subset S^{\prime}$.
 Hence, $x^{\prime}\notin L_{3}(C)$ because $x$ is in 
 $N_D(x')\setminus C$, a contradiction. Thus, $x' \in
 C\setminus(N_D(c')
\cup(S'\setminus A))\subset N^+_D(a)$. This implies $L_3(C) \subset N^+_D(a)$.

The next step is to prove that $C$ is a strong vertex cover. As 
$\{a,b^{\prime}\}, \{b,c^{\prime}\}\in E(G)$ and $\{b,b^{\prime}\}\in
P$, by condition $(1)$ we get $\{a,c^{\prime}\}\in E(G)$. Thus,
 $a\in N_{D}(c^{\prime})\subset C$. Also, $a\notin L_{1}(C)$, since
 $N_{D}^{+}(a)\subset C$. Hence, $C$ is a strong vertex cover, since
 $L_{3}(C)\subset N_{D}^{+}(a)$, $a\in L_2(C)\cup L_3(C)$, and
 $w(a)>1$. 
 
 Then, by Theorem \ref{theorem42} and Remark
 \ref{1star}, $|C|=\tau(G)=|P|$, since $I(D)$ is unmixed and $G$ is
 very well-covered. Then, $|C\cap e|=1$ for $e\in P$. This is a
 contradiction, since $b^{\prime},b\in N_{D}^{+}(a)\cup
 N_{D}(c^{\prime})\subset C$ and $\{b,b^{\prime}\}\in P$. Therefore,
$N_{D}(b)\subset N_{D}^{+}(a)$ and condition (2) holds.             

$\Leftarrow)$ By $(1)$ and Theorem \ref{konig}, $G$ is very
well-covered. Thus, by Remark \ref{1star}, $I(G)$ is unmixed. Now,
let $C$ be a strong vertex cover of $D$. Suppose
$L_{3}(C)\neq\emptyset$, then there is $b^{\prime}\in L_{3}(C)$. By
Lemma~\ref{Rem-L3}, $N_{D}(b^{\prime})\subset C$. Since $C$ is
strong, there is $a\in (V^{+}\cap C)\setminus L_{1}(C)$ such that
$(a,b^{\prime})\in E(D)$. Since $a\notin L_{1}(C)$, we have
$N_{D}^{+}(a)\subset C$. Furthermore, by (2), $N_{D}(b)\subset N_{D}^{+}(a)$       
 where $\{b,b^{\prime}\}\in P$. Hence, $b\in C$ and $b\in L_{3}(C)$, 
 since $N_{D}(b^{\prime})\cup N_{D}(b)\subset C$. By
 Lemma~\ref{remark-(ii)}, $N^{-}_{D}(b)\cap V^{+}=\emptyset$ because
 $\{b,b'\}\in P$ and $b'\in N_D(b)\subset N_D^+(a)$. This is
 a contradiction, 
since $C$ is strong and $b\in L_{3}(C)$. 
This implies, $L_{3}(C)=\emptyset$ for any strong vertex cover $C$ of
$D$. Therefore, by Theorem \ref{theorem42}, $I(D)$ is unmixed, since
$I(G)$ is unmixed.  
 \end{proof}

\begin{remark}\rm\label{azul}
By the proof of $\Rightarrow)$ in Theorem
\ref{prop-unmixed}, 
we have that: If $I(D)$ is unmixed and $P$ is a perfect matching of
$G$ with the 
property {\bf{(P)}}, then $P$ satisfies $(2)$ of Theorem \ref{prop-unmixed}. 
\end{remark}

If $G$ is an unmixed K\"onig simple hypergraph without isolated vertices, 
then $G$ has a perfect matching $P$ with $|P|=\tau(G)=\nu(G)$
\cite[Lemma~2.3]{susan-reyes-vila}, and $G$ is very 
well-covered if $G$ is a graph. The next result gives an analogous
version for weighted oriented graphs.

\begin{corollary}\label{jun24-19}
Let $D$ be a weighted oriented graph whose underlying graph $G$ is
K\"onig. Then $I(D)$ is unmixed if and only if $D$ satisfies the
following two conditions: 
\begin{enumerate}[{\rm (1)}]
\item $G$ is very well-covered.
\item If $\{b,b^{\prime}\}$ is in a perfect matching $P$ of $G$, 
$b^{\prime}\in N_{D}^{+}(a)$ and $w(a)>1$, then
$N_{D}(b)\subset N_{D}^{+}(a)$.
\end{enumerate}
\end{corollary}

\begin{proof} $\Rightarrow$) Assume that $I(D)$ is unmixed. Then, 
by Theorem \ref{prop-unmixed}(1) and Theorem \ref{konig}, $G$ is very
well-covered, that is conditions (1) holds. That condition (2) holds follows
from Remark \ref{azul} and Theorem \ref{konig}.

$\Leftarrow$) This implication follows using
Theorems~\ref{prop-unmixed} and \ref{konig}.
\end{proof}

The next result shows that Theorem~\ref{prop-unmixed} also holds when
$G$ has no $3$-, $5$-, or $7$-cycles.

\begin{proposition}\label{Coro-37}
Let $D$ be a weighted oriented graph whose underlying graph $G$  
has no $3$-, $5$-, or $7$-cycles. Then $I(D)$ is unmixed if and only
if $D$ satisfies the following conditions:  
\begin{enumerate}[{\rm (1)}]
\item $G$ has a perfect matching $P$ with property {\bf(P)}.
\item If $a\in V(D)$, $w(a)>1$, $b^{\prime}\in N_{D}^{+}(a)$, and $\{b,b^{\prime}\}\in P$, then $N_{D}(b)\subset N_{D}^{+}(a)$. 
 \end{enumerate}
\end{proposition}
 
 \begin{proof} 
 $\Rightarrow)$
By Theorem \ref{theorem42}, $I(G)$ is unmixed. Thus, by Remark
\ref{1star} and Theorem \ref{lemma-Ivan}, $G$ is very well-covered.
Hence, by Remark \ref{2star}, $G$ is K\"{o}nig. Therefore by Theorem
\ref{prop-unmixed}, $D$ satisfies $(1)$ and $(2)$.     

$\Leftarrow)$ 
By (1) and Theorem \ref{konig}, $G$ is very well-covered. Hence, by
Remark \ref{2star}, $G$ is K\"{o}nig. Therefore, by Theorem
\ref{prop-unmixed}, $D$ is unmixed.  
 \end{proof}
 
 \begin{corollary}\label{sep15-19}
 If $D$ is a weighted oriented graph whose underlying graph $G$ is 
 very well-covered, then $G$ has a perfect matching $P$ and 
the following conditions are equivalents: 
 \begin{enumerate}[{\rm (a)}]
\item $I(D)$ is unmixed.
\item If $a\in V(D)$, $w(a)>1$, $b^{\prime}\in N_{D}^{+}(a)$, and 
$\{b,b^{\prime}\}\in P$, then $N_{D}(b)\subset N_{D}^{+}(a)$.
 \end{enumerate}
 \end{corollary}
 \begin{proof}
 By Remark \ref{2star}, $G$ is K\"{o}nig. Furthermore, by Theorem
 \ref{konig}, $G$ has a perfect matching $P$ with property
 {\bf{(P)}}.  

 (a)$\Rightarrow$ (b): By Remark \ref{azul}, $P$ satisfies (b).

(b)$\Rightarrow$(a): By Theorem~\ref{konig} $D$ satisfies
condition (1) of Theorem~\ref{prop-unmixed}. Then, by Theorem
\ref{prop-unmixed}, $I(D)$ is unmixed, 
since $D$ satisfies (b).
  \end{proof}

\section{Cohen--Macaulay weighted oriented graphs}\label{c-m-section}

In this section we classify the Cohen--Macaulay property of a weighted
oriented graph $D$ whose underlying graph $G$ is K\"onig or $G$ is a
graph without $3$- and $5$-cycles. If
$G$ is a K\"onig graph without $4$-cycles or
$G$ has girth greater than $7$, we prove 
that $I(D)$ is unmixed if and only if $I(D)$ is 
Cohen--Macaulay. 

\begin{definition}
Let $D$ be a weighted oriented graph and $A\subset V(D)$, then $D\setminus A$ is the weighted oriented graph where $V(D\setminus A)=V(D)\setminus A$, $E(D\setminus A)=\{(a,b)\in E(D)\mid \{a,b\}\cap A=\emptyset \}$ and $w_{D}(x)=w_{D\setminus A}(x)$ for $x\in V(D\setminus A)$.
\end{definition}

The following result is well-known. It follows from the depth lemma
\cite[Lemma 2.3.9]{monalg-2}.

\begin{proposition}\label{propo-estrella2}
Let $I\subset R$ be a graded ideal and let $f$ be a
homogeneous polynomial of $R$ which is a zero-divisor of $R/I(D)$. 
The following hold.
\begin{enumerate}[{\rm (1)}]
\item If $I$ is unmixed and $f\notin I$, then $(I\colon f)$ is unmixed
and ${\rm ht}(I)={\rm ht}(I\colon f)={\rm ht}(I,f)$.
\item If $R/(I,f)$ and $R/(I\colon f)$ are Cohen--Macaulay, then $R/I$ is Cohen--Macaulay.
\end{enumerate}
\end{proposition}

We come to the main result of this section.

\begin{theorem}\label{theoremCM}
Let $D=(V(D),E(D),w)$ be a weighted oriented graph whose underlying graph $G$ is
K\"{o}nig. Then, $I(D)$ is Cohen--Macaulay if and only if $D$
satisfies the following two conditions:  
\begin{enumerate}[{\rm (1)}]
\item $G$ has a perfect matching $P$ such that if $\{a,b\}$,
$\{a^{\prime},b^{\prime}\}\in E(G)$, and $\{b,b^{\prime}\}\in P$, then
$\{a,a^{\prime}\}\in E(G)$. Furthermore $G$ has no $4$-cycles 
with two edges in $P$.
\item If $a\in V(D)$, $w(a)>1$, $b^{\prime}\in N_{D}^{+}(a)$, 
and $\{b,b^{\prime}\}\in P$, then $N_{D}(b)\subset N_{D}^{+}(a)$.
 \end{enumerate}
\end{theorem}

\begin{proof} 
$\Rightarrow)$ By Proposition \ref{teorema-estrella}, $I(G)$ is
Cohen--Macaulay. So, by Theorem \ref{3-5cycles}, $G$ satisfies (1). As
$I(D)$ is Cohen--Macaulay, $I(D)$ is unmixed. Hence, by
Remark \ref{azul}, 
$D$ satisfies (2).

$\Leftarrow)$ By induction on $|P|$. 
By Theorem \ref{3-5cycles}, $I(G)$ is Cohen--Macaulay.
Consequently, by Lemma~\ref{corolario-grado1}, there is $x^{\prime}\in
V(D)$ such that $\deg_{D}(x^{\prime})=1$. Then, $x^\prime$ is a
source or a sink. Hence $w(x^{\prime})=1$. 
Since $P$ is perfect, there is
$x\in V(D)$ such that $e':=\{x,x^{\prime}\}\in P$.   

One has the equality $(I(D),x)=(I(D_1),x)$, where 
$D_1=D\setminus\{x,x^{\prime}\}$. 
We denote the underlying graph of $D_1$ by $G_1$. 
Thus $Q_{1}:=P\setminus\{e'\}$ is a
 perfect matching of $G_1$ such that, 
$Q_1$ satisfies (1)
and (2) in $D_1$, since $P$ satisfies them in $D$. Then, by Theorem \ref{konig}, $G_1$ is very well-covered, since $Q_1$ satisfies (1). So, by Remark \ref{2star}, $G_1$ is K\"onig. Hence, by
 induction $I(D_1)$ is Cohen--Macaulay. This implies $(I(D),x)$ is
 Cohen--Macaulay because $(I(D),x)=(I(D_1),x)$ and $x$ is regular
 modulo $I(D_1)$.

Furthermore, the variable $x$ is a zero-divisor of $R/I(D)$, since either $xx'$ or $x'x^{w(x)}$ is a minimal generator of $I(D)$. Therefore, by Proposition \ref{propo-estrella2}(2), to prove 
that $I(D)$ is Cohen--Macaulay it suffices to prove that $J:=(I(D)\colon x)$
is Cohen--Macaulay.

For $i=1,2$, we set $V^{i}:=\{z\mid w(z)=i\}$, 
$$V^{\prime}:=N_{D}^{+}(x)\cap V^{1}\ \mbox{ and }\ 
V^{\prime\prime}:=N_{D}^{+}(x)\cap V^{2}.
$$
\quad By Proposition~\ref{propo-estrella}, we may assume 
$V^{+}=V^{2}$ and $V(D)=V^{1}\cup V^{2}$.  We consider the cases
$w(x)=2$ and $w(x)=1$. As is seen below in the first case $x$ is a zero-divisor of
$R/J$ and in the second case $x$ is a regular element of $R/J$.

Case (A): Assume that $w(x)=2$. If $(x,x^{\prime})\in E(D)$, then
$x^{\prime}\in N_{D}^{+}(x)$. So, by (2), one has $N_{D}^{}(x)\subset
N_{D}^{+}(x)$, since $w(x)=2$ and $\{x',x\}\in P$. Thus
$N_{D}^{-}(x)=\emptyset$, i.e., $x$ is a source. A contradiction,
since $x\in V^{+}$. Then, $(x^{\prime},x)\in E(D)$ and $x^{\prime}\in N_{D}^{-}(x)$.  
So, $x^{\prime}x^2$ is a minimal generator of $I(D)$ and
$x^{\prime}x$ is a minimal generator of $J$. This implies, $x\not\in
J$ and $x$ is a zero-divisor of $R/J$ because $x^2\notin I(D)$ and
$x^{\prime}x\in J$. Hence, by Proposition~\ref{propo-estrella2}(2),
we only need to show     
that $(J\colon x)$ and $(J, x)$ are 
Cohen--Macaulay. We can write $J$ as 
\begin{equation}\label{jul25-19}
J=(I(D)\colon x)=(xN_{D}^{-}(x),V^{\prime},\{v^{2}\mid v\in
V^{\prime\prime}\},I(D\setminus A_1)),
\end{equation}
where $A_1=V^{\prime}\cup \{x\}$, then $ x'\notin A_1$, since $x^{\prime}\in N_{D}^{-}(x)$. Using
Eq.~(\ref{jul25-19}), we get the equalities
\begin{eqnarray}\label{Aug28-19}
(J, x)&=&(A_1,
\{v^{2}\mid v\in V^{\prime\prime}\}, I(D\setminus A_1)),
\end{eqnarray} 
\begin{eqnarray}\label{Aug28-19-2}
(J\colon x)&=&(A_2\setminus \{x\},\{v^{2}\mid v\in V^{\prime
\prime}\}, I(D\setminus A_2)),
\end{eqnarray}
where $A_2=N^{-}_{D}(x)\cup V^{\prime}\cup\{x\}$. Note that $x'\in
A_2$. Setting $L_i:=(\{v^{2}\mid v\in
V^{\prime\prime}\}, I(D\setminus A_i))$ for $i=1,2$, from
Eqs.~(\ref{Aug28-19}) and 
(\ref{Aug28-19-2}) we get
$$(J, x)=(A_1,L_1)\ \mbox{ and }\ (J\colon x)=(A_2\setminus\{x\},L_2).
$$
\quad Hence, we only need to prove that $L_1$ and $L_2$ are Cohen--Macaulay, since $A_1$
and $A_2\setminus\{x\}$ are regular sequences of $R/L_1$ and 
$R/L_2$, respectively. To show this, we will consider for $i=1,2$ the
following auxiliary graphs $\mathcal{H}_i$ and $\mathcal{F}_i$. 

For $i=1,2$ consider the weighted oriented graph $\mathcal{H}_i$ with
$V(\mathcal{H}_i):=V(D\setminus A_i)$ and  
$$
E(\mathcal{H}_i):=E(D\setminus A_i)\setminus \{(a_1,a_2)\in E(D)\mid
a_2\in V^{\prime\prime}\}.
$$
\quad Consequently, the elements of $V^{\prime\prime}$ are sources in
$\mathcal{H}_i$ and $I(\mathcal{H}_i)\subset I(D\setminus A_i)$ for
$i=1,2$. Furthermore, if $e\in E(D\setminus A_i)\setminus
E(\mathcal{H}_i)$, then $e=(a_1,a_2)$ with $a_2\in V^{''}$ implying
$x_e:=a_{1}a_{2}^{w(a_2)}=a_{1}a_{2}^{2}\in (\{v^{2}\mid v\in
V^{''}\})$. Hence, $I(D\setminus A_i)\subset
(I(\mathcal{H}_i),\{v^{2}\mid v\in V^{\prime\prime}\})$      
 and we get
$$ 
(I(\mathcal{H}_i),\{v^{2}\mid
v\in V^{\prime\prime}\})\subset ( I(D\setminus A_i),\{v^{2}\mid v\in
V^{\prime\prime}\})\subset (I(\mathcal{H}_i),\{v^{2}\mid
v\in V^{\prime\prime}\}),
$$
since $I(\mathcal{H}_i)\subset I(D\setminus A_i)$. Therefore, 
$L_i=(I(\mathcal{H}_i),\{v^{2}\mid v\in V^{\prime\prime}\})$ for
$i=1,2$. Now, setting         
\begin{center}
$\Omega_i:=\{a\in V(\mathcal{H}_i)\mid \{a,a^{\prime}\}\in P
 \textit{ with } a^{\prime}\in A_{i}\cup V^{''}\}$,
\end{center}
we will prove that  $\Omega_i$ is a set of isolated vertices in $\mathcal{H}_i$ for
 $i=1,2$, that is, $N_{\mathcal{H}_i}(a)=\emptyset$ for
 $a\in\Omega_i$. 
 We take $a\in \Omega_i$,
 then there is $\{a,a^{\prime}\}\in P$ with $a^{\prime}\in A_{i}\cup V^{''}
\subset N_{D}^{}(x)\cup\{x\}$. If $a^{\prime}=x$, then $a=x^{\prime}$. Thus, $a$ is isolated in $\mathcal{H}_i$, since $x\in A_i$ and ${\rm deg}_{D}(x^{\prime})=1$. Consequently, we can assume $a^{\prime}\in N_{D}(x)$. By contradiction, suppose there is $b\in V(\mathcal{H}_i)$ such that
 $b\in N_{\mathcal{H}_i}(a)$. By (1), $b\in N_{D}(x)$, since $a^{\prime}\in N_{D}(x)$. So, if $i=2$, then $b\in V^{''}$, since $b\not\in A_{2}$. Now, if $i=1$, then $a^{\prime}\in (A_1\cup V^{''})\setminus \{x\}\subset N_{D}^{+}(x)$. Hence, by (2), $b\in N_{D}^{+}(x)$. But $b\notin
 A_1$, then $b\in V^{\prime\prime}$. Thus, in both cases $b\in V^{\prime\prime}$, that is, $x\in N^{-}_{D}(b)$  and
 $w(b)=2$. By definition of $E(\mathcal{H}_{i})$, $(a,b)\notin
 E(\mathcal{H}_i)$, since $b\in V^{\prime\prime}$. Then $(b,a)\in E(\mathcal{H}_i)$. This implies, $a\in
 N_{D}^{+}(b)$ and by (2), we have $N_{D}(a^{\prime})\subset N_{D}^{+}(b)$. In
 particular, $x\in N_{D}^{+}(b)$, since $x\in N_{D}(a^{\prime})$. A
 contradiction, since $x\in N^{-}_{D}(b)$. 
Therefore $\Omega_i$ is a
 set of isolated vertices in $\mathcal{H}_i$.                      

Note that $V^{''}\subset V(D)\setminus A_i=V(\mathcal{H}_i)$. We will prove that $V^{\prime\prime}\cap \Omega_i=\emptyset$. By contradiction suppose $a\in V^{''}\cap\Omega_i$, then $a^{\prime}\in A_i\cup V^{''}\subset N_{D}(x)\cup\{x\}$ where $\{a,a^{\prime}\}\in P$. If $a^{\prime}=x$, then $x^{\prime}=a\in V^{''}$ implies $(x,x^{\prime})\in E(D)$. A contradiction, since $(x^{\prime},x)\in E(D)$. Consequently, $a^{\prime}\in N_{D}(x)$, but $a\in V^{''}\subset N_{D}(x)$. This is a contradiction by (1). Therefore $V^{''}\cap \Omega_i=\emptyset$ and $V^{''}\subset V(\mathcal{H}_i\setminus \Omega_i)$. 

If $\{a,a^{\prime}\}\in P$ and $a^{\prime}\in V^{\prime\prime}$, then $a\in \Omega_i$ and $a\not\in V(\mathcal{H}_i\setminus \Omega_i)$. Now, for each $v\in V^{\prime\prime}$ let $y_v$ be a new variable and consider the weighted oriented graph $\mathcal{F}_i$ for $i=1,2$,
 whose vertex set and edge set are
\begin{eqnarray*}
 V(\mathcal{F}_i)&:=&V(\mathcal{H}_i\setminus \Omega_i)\cup \{y_v\mid
 v\in V^{\prime\prime} \},\\
 E(\mathcal{F}_i)&:=&E(\mathcal{H}_i\setminus \Omega_i)\cup
 \{(v,y_v)\mid v\in V^{\prime\prime}\},
\end{eqnarray*}
respectively, and whose weight function
 $w_i$, is given by 
$$
w_i(u):=
\begin{cases}
w(u)&if\ u\in V(\mathcal{H}_i\setminus\Omega_i)\mbox{ and }u\notin
V^{\prime\prime}, \\
1&if\ u\in V^{\prime\prime}\cup\{y_v\vert\, v\in V^{\prime\prime}\}. 
\end{cases}
$$
\quad 
Since $\Omega_i$ is a set of isolated vertices in $\mathcal{H}_i$, we 
have $E(\mathcal{F}_i)=E(\mathcal{H}_i)\cup\{(v,y_v)\mid v\in
V^{\prime\prime}\}$. Also, if $e_1=(a_1,a_2)\in E(\mathcal{H}_i)$,
then $a_2\not\in V^{\prime\prime}$, by definition of
$E(\mathcal{H}_i)$. Thus,
$m_{e_1}:=a_1a_2^{w(a_2)}=a_1a_2^{w_i(a_2)}\in I(\mathcal{F}_i)$.
Similarly if $e_2\in E(\mathcal{F}_i)\setminus \{(v,y_v)\mid v\in
V^{\prime\prime}\}$, then $e_2=(a_1,a_2)$ with $a_2\not\in
V^{\prime\prime}$, implies $a_1a_2^{w_i(a_2)}=a_1a_2^{w(a_2)}\in
I(\mathcal{H}_i)$. Consequently,
$$I(\mathcal{F}_i)=(I(\mathcal{H}_i),\{vy_v\mid v\in
V^{\prime\prime}\}).$$ 
Furthermore, $L_i=(I(\mathcal{H}_i),\{v^2\mid
v\in V^{\prime \prime}\})$. Then $I(\mathcal{F}_i)$ is a partial polarization of $L_i$
 obtained from $L_i$ by polarizing all monomials $v^2$ with $v\in
 V^{\prime\prime}$, that is, we replace $v^2$ by $vy_v$. 
 Hence, we only need to prove that $I(\mathcal{F}_{i})$ is Cohen--Macaulay for $i=1,2$ (cf. \cite[p.~555]{reyes-vila}). 
  
Let $F_i$ be the underlying graph of $\mathcal{F}_i$ for $i=1,2$. We
will prove that  
\begin{center}
$P_{i}:=(P\setminus (\{e\in P \mid e\cap \Omega_i\neq \emptyset\}\cup \{e^{\prime}\}))\cup
\{\{v,y_v\}\mid v\in V^{\prime \prime}\}$
\end{center} 
is a perfect matching in $F_i$, where $e^{\prime}=\{x,x^{\prime}\}$.
We suppose $\tilde{e}\in P\setminus (\{e\in P \mid e\cap \Omega_i\neq
\emptyset\}\cup \{e^{\prime}\})$ such that $\tilde{e}\cap V^{''}\neq
\emptyset$. Thus, $\tilde{e}=\{a,a^{\prime}\}$ with $a^{\prime}\in
V^{''}$. Consequently, $a\in\Omega_i$, so $\tilde{e}\cap\Omega_i\neq
\emptyset$. A contradiction, then $P_i$ is a matching, since $P$ is a
matching and $\{y_v\mid v\in V^{''}\}\cap V(D)=\emptyset$. So, to
show that $P_i$ is a perfect matching of $F_i$, we need only show 
that the following equality holds         
$$
V(F_i):=(V(D)\setminus(A_i\cup\Omega_i))\cup \{y_v\mid
 v\in V^{\prime\prime} \}=\bigcup_{e\in P_i}e.
$$

We take $e=\{a,a^{\prime}\}\in P_i$. If $e=\{v,y_v\}$ with $v\in
V^{''}$, then $e\subset V(\mathcal{F}_i)$, since $V^{''}\subset
V(\mathcal{F}_i)$. Now, we assume $e\in P$, then
$e\cap\Omega_i=\emptyset$ and $e\neq e^{\prime}$. Thus,
$x\not\in\{a,a^{\prime}\}$ and by (1) we have $|\{a,a^{\prime}\}\cap
A_i|\leq 1$, since $A_i\setminus\{x\}\subset N_{D}(x)$. We can
assume $a\not\in A_i$, then $a\in V(\mathcal{H}_i)$. But $a\not\in
\Omega_i$, then $a^{\prime}\not\in A_i\cup V^{''}$. So, $e\cap
A_i=\emptyset$. Hence, $e\subset V(\mathcal{F}_i)$, since
$e\cap\Omega_i=\emptyset$. Now, to show the inclusion $``\subset"$,
we take $b\in V(\mathcal{F}_i)$. Then, $b\neq x$, since $x\in A_i$.
If $b=x^{\prime}$, then $x^{\prime}\in V(\mathcal{H}_i)$.
Consequently, $x^{\prime}\in \Omega_i$, since $x\in A_i$. A
contradiction, since $b\not\in \Omega_i$. This implies $b\neq
x^{\prime}$. If $b\in \{y_v,v\}$ for some $v\in V^{''}$, then $b\in
\bigcup_{e\in P_i}e$. Now, we can assume 
$$b\not\in V^{''}\cup
\{y_v\mid v\in V^{\prime\prime} \}\cup \{x,x^{\prime}\}\ \mbox{ and }\ b\in
V(D)\setminus (A_i\cup \Omega_i).$$
Thus, there is
$\tilde{e}=\{b,b^{\prime}\}\in P$ such that $\tilde{e}\neq
e^{\prime}$, since $P$ is a perfect matching of $D$. Hence,
$b^{\prime}\not\in \Omega_i$, since $b\not\in V^{''}\cup A_i$.
Consequently $\tilde{e}\cap\Omega_i=\emptyset$, since $b\not\in
\Omega_i$. Therefore, $\tilde{e}\in P_i$ and $b\in \bigcup_{e\in
P_i}e$, since $\tilde{e}\neq e^{\prime}$.                       

Next we show that $P_i$ satisfies (1) and (2) in $\mathcal{F}_i$. Assume that $a\in V(\mathcal{F}_i)$, $\{a,b'\}\in E(F_i)$ and $\{b,b^{\prime}\}\in P_{i}$. We will prove $N_{\mathcal{F}_i}(b)\subset N_{\mathcal{F}_i}^{}(a)$; furthermore if $a\in V^{+}(\mathcal{F}_i)$ and $(a,b')\in E(\mathcal{F}_i)$, then we will show  $N_{\mathcal{F}_i}(b)\subset N_{\mathcal{F}_i}^{+}(a)$. If $b=y_v$, then $b^{\prime}=v$ and
$$N_{\mathcal{F}_i}(b)=N_{\mathcal{F}_i}^{}(y_v)=\{b^{\prime}\}\subset N_{\mathcal{F}_i}^{}(a).$$
Also, if $(a,b^{\prime})\in E(\mathcal{F}_i)$, then
$N_{\mathcal{F}_i}^{}(b)=\{b^{\prime}\}\subset
N_{\mathcal{F}_i}^{+}(a)$. Now, if $b\in V^{\prime\prime}$, then
$b'=y_b$. Consequently, $a=b$, $N_{\mathcal{F}_i}(b)\subset
N_{\mathcal{F}_i}^{}(a)$ and $w_{i}(a)=1$ (i.e. $a\not\in
V^{+}(\mathcal{F}_i)$), since ${\rm deg}_{\mathcal{F}_i}(y_b)=1$ and
$a=b\in V^{\prime\prime}$. Thus, we may assume $b\in V(D)\setminus
V^{\prime\prime}$, then $b^{\prime}\in V(D\setminus A_i)$, since
$P_i$ is a perfect matching of $F_i$. If $a=y_v$, then $b^{\prime}=v$
and $b=y_v$. A contradiction, since $b\in V(D)$. Consequently, $a\in
V(D)$. Furthermore, $N_{\mathcal{F}_i}(b)\subset
V(\mathcal{H}_i\setminus \Omega_i)$ since $b\not\in
V^{\prime\prime}$. As $P$ satisfies (1) and by definition of
$E(\mathcal{H}_i)$, we have             
$$
N_{\mathcal{F}_i}(b)\subset N_{D}(b)\cap V(\mathcal{H}_i\setminus \Omega_i)\subset N_{D}^{}(a)\cap V(\mathcal{H}_i\setminus \Omega_i)\subset N_{\mathcal{F}_i}^{}(a)\cup \{a_2\in V^{\prime\prime}\mid (a,a_2)\in E(D)\}.
$$

Suppose, $c\in N_{\mathcal{F}_i}(b)\cap \{a_2\in V^{\prime\prime}\mid
(a,a_2)\in E(D)\}$, then $(a,c)\in E(D)$ and $\{c,b\}\in E(F_i)$. But
$V^{\prime\prime}$ is a set of source vertices in $\mathcal{H}_i$ and
$c\in V^{\prime\prime}$, then $(c,b)\in E(\mathcal{F}_i)\subset
E(D)$. Consequently $w(c)>1$ and $b\in N_{D}^{+}(c)$. This implies
$N_{D}(b^{\prime})\subset N_{D}^{+}(c)$, since $\{b,b^{\prime}\}\in
P$ and $D$ satisfies (2). Hence, $a\in N_{D}(b^{\prime})\subset
N_{D}^{+}(c)$. A contradiction, since $(a,c)\in E(D)$. Therefore
$$N_{\mathcal{F}_i}(b)\cap \{a_2\in V^{\prime\prime}\mid (a,a_2)\in
E(D)\}=\emptyset$$ 
and $N_{\mathcal{F}_i}(b)\subset
N_{\mathcal{F}_i}(a)$. Now, if $a\in V^{+}(\mathcal{F}_i)$ and
$(a,b^{\prime})\in E(\mathcal{F}_i)$, then $a\in V^{+}(D)$ and
$(a,b^{\prime})\in E(D)$. Thus $N_{D}(b)\subset N_{D}^{+}(a)$, since
$P$ satisfies (2). Hence             
$$N_{\mathcal{F}_i}(b)\subset
N_{D}(b)\cap V(\mathcal{H}_i\setminus \Omega_i)\subset
N_{D}^{+}(a)\cap V(\mathcal{H}_i\setminus \Omega_i)\subset
N_{\mathcal{F}_i}^{+}(a)\cup \{a_2\in V^{\prime\prime}\mid (a,a_2)\in E(D)\}.
$$

But $N_{\mathcal{F}_i}(b)\cap \{a_2\in V^{\prime\prime}\mid (a,a_2)\in E(D)\}=\emptyset$, then $N_{\mathcal{F}_i}(b)\subset N_{\mathcal{F}_i}^{+}(a)$. Furthermore, $F_i$ has no $4$-cycles with two edges in $P_i$, since $D$ has no $4$-cycle with two edges in $P$ and ${\rm deg}_{\mathcal{F}_i}(y_v)=1$ for each $v\in V^{\prime\prime}$. This implies, $P_i$ satisfies (1) and (2) in $\mathcal{F}_i$. 

By Theorem \ref{konig}, $F_i$ is very well-covered, since $P_i$ satisfies (1). Then, by Remark \ref{2star}, $F_i$ is K\"onig. Therefore, by induction hypothesis, $I(\mathcal{F}_i)$ is Cohen--Macaulay for $i=1,2$.

Case (B): Assume that $w(x)=1$. In this case we can
write $J$ as
$$J=(I(D)\colon   
 x)=(B_{1},\{b^2\mid b\in V^{\prime\prime}\}, I(D\setminus B_{1})),$$
 where $B_{1}=V^{\prime}\cup N_{D}^{-}(x)$. Consequently, $x$ is regular on $R/J$
because $x$ is not in any minimal monomial generator of $J$. Then,
$R/J$ is Cohen--Macaulay if and only if $R/(J,x)$ is Cohen--Macaulay.
Thus, it suffices to show that $(J,x)$ is Cohen--Macaulay.

Furthermore, $(J,x)=(B, \{v^{2}\mid v\in V^{\prime\prime}\}, I(D\setminus B))$, 
where $B=B_1\cup \{x\}$. Thus, to prove that $(J,x)$ is 
Cohen--Macaulay, it is only necessary to prove that the ideal 
$$L:=(\{v^{2}\mid v\in V^{\prime\prime}\}, I(D\setminus B))$$ 
is Cohen--Macaulay, since
 $B$ is a regular sequence in $R/L$. 
But $L=L_2$ in Case (A), then with the same arguments on $L_2$, 
it follows that $L$ is Cohen--Macaulay.       
 \end{proof}

\begin{corollary}\label{jun23-19}
Let $D$ be a weighted oriented graph, where $G$ is a K\"onig graph without $4$-cycles. Hence, $I(D)$ is unmixed if and only if $I(D)$ is Cohen--Macaulay.
\end{corollary}
\begin{proof}
It follows from Theorems \ref{prop-unmixed} and \ref{theoremCM}. 
\end{proof}

The next result shows that Theorem~\ref{theoremCM} also holds when
$G$ has no $3$- or $5$-cycles.

\begin{proposition}\label{Coro-45}
Let $D$ be a weighted oriented graph without $3$- and $5$-cycles, 
then $I(D)$ is Cohen--Macaulay if and only if $D$ satisfies the following two conditions:
\begin{enumerate}[{\rm (1)}] 
\item $G$ has a perfect matching $P$ with property {\bf(P)} and $G$ has no $4$-cycles 
with two edges in $P$.
\item If $a\in V(D)$, $w(a)> 1$, $b^{\prime}\in N_{D}^{+}(a)$ and $\{b,b^{\prime}\}\in P$, then $N_{D}(b)\subset N_{D}^{+}(a)$.
 \end{enumerate}
\end{proposition}
\begin{proof}
$\Rightarrow)$ By Proposition \ref{teorema-estrella}, $I(G)$ is
Cohen--Macaulay. Thus, by (b) in Theorem \ref{3-5cycles}, $G$ is very
well-covered. Hence, by Remark \ref{2star}, $G$ is K\"{o}nig.
Therefore, by Theorem \ref{theoremCM}, $D$ satisfies (1) and (2).   

$\Leftarrow)$ By Theorem \ref{konig}, $G$ is very well-covered, since
$G$ satisfies (1). Consequently, by Remark \ref{2star}, $G$ is
K\"{o}nig. Therefore $D$ is Cohen--Macaulay, by Theorem \ref{theoremCM}.   
\end{proof}

The following result proves \cite[Conjecture~53]{WOG},  
when $G$ is a K\"onig graph or $G$ is a graph without $3$- and $5$-cycles. 

\begin{corollary}\label{conjecture-WOG}
Let $D$ be a weighted oriented graph whose underlying graph $G$ is
K\"onig or $G$ has no $3$- or $5$-cycles. Then $I(D)$ is
Cohen--Macaulay if and only if $I(D)$ is unmixed and $I(G)$ is
Cohen--Macaulay.    
\end{corollary}
\begin{proof}
$\Rightarrow)$ It follows from Proposition \ref{teorema-estrella} and Theorem \ref{theorem42}.

$\Leftarrow)$ By Theorem \ref{3-5cycles}, $G$ satisfies (1)
of Theorem \ref{theoremCM}. Hence, by Theorem \ref{konig}, $G$ is
very well-covered. Thus, by Remark~\ref{2star}, $G$ is K\"onig. 
Consequently, by Remark \ref{azul}, $D$ satisfies (2) of Theorem
\ref{theoremCM}, since $I(D)$ is unmixed. Therefore, by Theorem \ref{theoremCM}, $I(D)$ is Cohen--Macaulay.   
\end{proof}

The girth of a graph $G$ is the length of a shortest cycle contained in $G$. If $G$ does not contain any cycles, its girth is
defined to be infinity. 

\begin{corollary}\label{jun25-19}
Let $D$ be a weighted oriented graph such that $G$ has girth greater than $7$. Hence, $I(D)$ is unmixed if and only if $I(D)$ is Cohen--Macaulay.
\end{corollary}
\begin{proof}
It follows from Propositions~\ref{Coro-37} and \ref{Coro-45}.
\end{proof}

\section{Examples}

\begin{example}
The two weighted oriented graphs depicted in Figure~\ref{Ejemplo2}
are mixed, and their underlying graphs are  unmixed K\"onig graphs
with a perfect matching. 
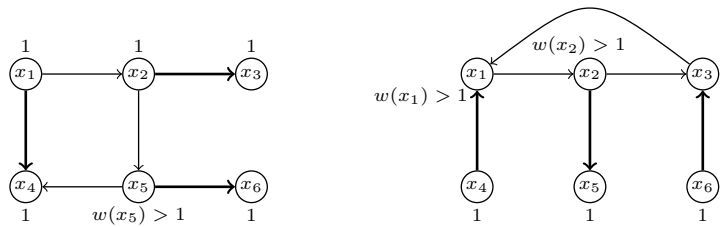
\begin{figure}[H]
\centering
\begin{tikzpicture}[line width=.5pt,scale=0.75]
		\tikzstyle{every node}=[inner sep=1pt, minimum width=5.5pt] 
\tiny{
\node[draw, circle] (2) at (-10,0){$x_{2}$};
\node[draw, circle] (5) at (-10,-2) {$x_{5}$};
\node[draw, circle] (3) at (-8,0) {$x_{3}$};
\node[draw, circle] (6) at (-8,-2){$x_{6}$};
\node[draw, circle] (1) at (-12,0){$x_{1}$};
\node[draw, circle] (4) at (-12,-2) {$x_{4}$};
\node at (-10,0.5){$1$};
\node at (-8,0.5) {$1$};
\node at (-8,-2.5){$1$};
\node at (-12,0.5){$1$};
\node at (-12,-2.5) {$1$};
\node at (-10,-2.5) {\tiny{$w(x_{5})>1$}};
\draw[->] (2) to (5);
\draw[->] (5) to (4);
\draw[->,line width=1pt] (1) -- (4);
\draw[->] (1) -- (2);
\draw[->,line width=1pt] (2) -- (3);
\draw[->,line width=1pt] (5) -- (6);
}
\tiny{
\node[draw, circle] (2) at (-2,0){$x_{2}$};
\node[draw, circle] (5) at (-2,-2) {$x_{5}$};
\node[draw, circle] (3) at (0,0) {$x_{3}$};
\node[draw, circle] (6) at (0,-2){$x_{6}$};
\node[draw, circle] (1) at (-4,0){$x_{1}$};
\node[draw, circle] (4) at (-4,-2) {$x_{4}$};
\node at (-2,-2.5) {\tiny{$1$}};
\node at (-4,-2.5) {\tiny{$1$}};
\node at (-5,-0.4) {\tiny{$w(x_{1})>1$}};
\node at (-2.2,0.5) {\tiny{$w(x_{2})>1$}};
\node at (.45,0) {\tiny{$1$}};
\node at (0,-2.5) {\tiny{$1$}};

\draw[->] (1) to (2);
\draw[->] (2) to (3);
\draw[->,line width=1pt] (4) to (1);
\draw[->,line width=1pt] (2) to (5);
\draw[->,line width=1pt] (6) -- (3);
\draw[->] (3) .. controls (-2,1.5) .. (1);

}
\end{tikzpicture}
\caption{$G$ is an unmixed graph and $I(D)$ is
mixed}\label{Ejemplo2}
\end{figure}
\end{example}

\begin{example}
The two weighted oriented graphs represented in Figure~\ref{Ejemplo3} are unmixed
and not Cohen--Macaulay, and their underlying graphs are
K\"onig and have a perfect matching. 
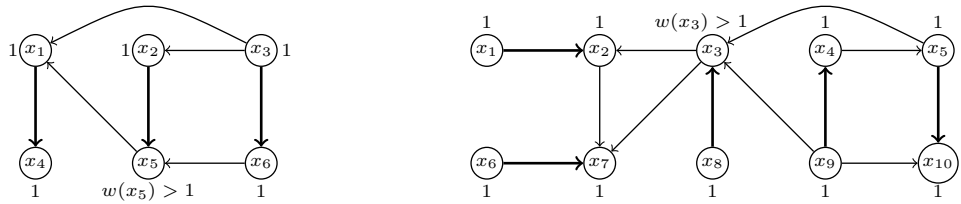
\begin{figure}[h]
\begin{center}
\begin{tikzpicture}[line width=.5pt,scale=0.75]
		\tikzstyle{every node}=[inner sep=1pt, minimum width=5.5pt] 
\tiny{
\node[draw, circle] (2) at (-2,0){$x_{2}$};
\node[draw, circle] (5) at (-2,-2) {$x_{5}$};
\node[draw, circle] (3) at (0,0) {$x_{3}$};
\node[draw, circle] (6) at (0,-2){$x_{6}$};
\node[draw, circle] (1) at (-4,0){$x_{1}$};
\node[draw, circle] (4) at (-4,-2) {$x_{4}$};
\node at (-2,-2.5) {\tiny{$w(x_{5})>1$}};
\node at (-4,-2.5) {\tiny{$1$}};
\node at (-4.4,0) {\tiny{$1$}};
\node at (-2.4,0) {\tiny{$1$}};
\node at (.45,0) {\tiny{$1$}};
\node at (0,-2.5) {\tiny{$1$}};

\draw[->,line width=1pt] (1) to (4);
\draw[->] (5) to (1);
\draw[->,line width=1pt] (2) to (5);
\draw[->] (3) to (2);
\draw[->,line width=1pt] (3) -- (6);
\draw[->] (6) -- (5);
\draw[->] (3) .. controls (-2,1) .. (1);

\node[draw, circle] (1) at (4,0){$x_{1}$};
\node[draw, circle] (6) at (4,-2) {$x_6$};
\node[draw, circle] (2) at (6,0) {$x_{2}$};
\node[draw, circle] (7) at (6,-2){$x_{7}$};
\node[draw, circle] (3) at (8,0){$x_{3}$};
\node[draw, circle] (8) at (8,-2) {$x_{8}$};
\node[draw, circle] (4) at (10,0) {$x_{4}$};
\node[draw, circle] (9) at (10,-2){$x_{9}$};
\node[draw, circle] (5) at (12,0){$x_{5}$};
\node[draw, circle] (10) at (12,-2) {$x_{10}$};
\node at (4,-2.5) {\tiny{$1$}};
\node at (6,-2.5) {\tiny{$1$}};
\node at (8,-2.5) {\tiny{$1$}};
\node at (10,-2.5) {\tiny{$1$}};
\node at (12,-2.5) {\tiny{$1$}};
\node at (4,.5) {\tiny{$1$}};
\node at (6,.5) {\tiny{$1$}};
\node at (7.8,.5) {\tiny{$w(x_{3})> 1$}};
\node at (10,.5) {\tiny{$1$}};
\node at (12,.5) {\tiny{$1$}};
\draw[->,line width=1pt] (1) to (2);
\draw[->] (3) to (2);
\draw[->,line width=1pt] (6) to (7);
\draw[->] (2) to (7);
\draw[->] (3) to (7);
\draw[->,line width=1pt] (8) to (3);
\draw[->] (9) to (3);
\draw[->,line width=1pt] (9) to (4);
\draw[->] (9) to (10);
\draw[->,line width=1pt] (5) to (10);
\draw[->] (4) to (5);
\draw[->] (5) .. controls (10,1) .. (3);
}
\end{tikzpicture}
\caption{$I(D)$ is unmixed and is not Cohen--Macaulay.}\label{Ejemplo3}
\end{center}
\end{figure}
\end{example}

\begin{example} The weighted oriented graph $D$ of
Figure~\ref{Ejemplo1} is Cohen--Macaulay, has an underlying graph $G$ which is K\"onig and
has a perfect matching. 
\begin{figure}[h]
\begin{center}
\begin{tikzpicture}[line width=.5pt,scale=0.75,rotate=90]
		\tikzstyle{every node}=[inner sep=1pt, minimum width=5.5pt] 
\tiny{
\node[draw, circle] (1) at (-2,-3){$x_{1}$};
\node[draw, circle] (2) at (0,-3) {$x_{2}$};
\node[draw, circle] (5) at (-2,-5) {$x_{5}$};
\node[draw, circle] (3) at (0,-5){$x_{3}$};
\node[draw, circle] (4) at (-1,-7){$x_{4}$};
\node[draw, circle] (6) at (-2,-1){$x_{6}$};
\node[draw, circle] (7) at (0,-1){$x_{7}$};
\node[draw, circle] (8) at (0,-7){$x_{8}$};

\draw[->,line width=1pt] (7) to (2);
\draw[->,line width=1pt] (6) to (1);
\draw[->] (2) to (5);
\draw[->] (1) to (3);
\draw[->] (1) to (5);
\draw[->,line width=1pt] (5) -- (4);
\draw[->] (1) --(2);
\draw[->] (3) -- (4);
\draw[->] (2) to (3);
\draw[->,line width=1pt] (3) -- (8);

\node at (-2.5,-3) {\tiny{$w(x_{1})>1$}};
\node at (.5,-3) {\tiny{$w(x_{2})>1$}};
\node at (-2.5,-5) {\tiny{$1$}};
\node at (-1.5,-7) {\tiny{$1$}};
\node at (0.5,-5) {\tiny{$1$}};
\node at (0.5,-7) {\tiny{$1$}};
\node at (-2.5,-1) {\tiny{$1$}};
\node at (0.5,-1) {\tiny{$1$}};
}
\end{tikzpicture}
\caption{$I(D)$ is Cohen--Macaulay}\label{Ejemplo1}
\end{center}
\end{figure}
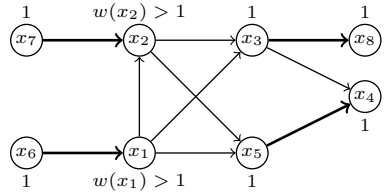
\end{example}


\bibliographystyle{amsplain}

\begin{thebibliography}{10}

\bibitem{digraphs} J. Bang-Jensen and G. Gutin, {\it Digraphs}. {\it Theory,
Algorithms and Applications\/}, Springer Monographs in Mathematics,
Springer, 2006.

\bibitem{Ivan-Reyes} I. D. Castrill\'on, R. Cruz and E. Reyes, On
well-covered, vertex decomposable and Cohen--Macaulay graphs, 
Electron. J. Combin. {\bf 23} (2016), no. 2, Paper 39, 17 pp.

\bibitem {EV}{M. Estrada and R. H. Villarreal, Cohen-Macaulay 
bipartite graphs, Arch. Math. {\bf 68} (1997), 124--128.}

 \bibitem{favaron}  O. Favaron, Very well-covered graphs, Discrete Math. {\bf 42} (1982), 
177--187.
 
 \bibitem{vivares} P. Gimenez, J. Mart\'inez-Bernal, A. Simis, R. H.
 Villarreal and 
C. E. Vivares, Symbolic powers of monomial ideals and
Cohen--Macaulay vertex-weighted digraphs, in {\it Singularities,
Algebraic Geometry, Commutative Algebra, and 
Related Topics\/} (G. M. Greuel, et.al. Eds), Springer, Cham, 2018, 
pp. 491--510. 

\bibitem{graphs-rings} I. Gitler and R. H. Villarreal, 
{\it Graphs, Rings and Polyhedra\/}, Aportaciones Mat. 
Textos, {\bf 35}, Soc. Mat. Mexicana, M\'exico, 2011.   

\bibitem{mac2} D. Grayson and M. Stillman, 
{\em Macaulay\/}$2$, 1996. 
Available via anonymous ftp from {\tt math.uiuc.edu}.

\bibitem{reyes-vila} H. T. H\'a, K. N. Lin, S. Morey, E. Reyes and
R. H. Villarreal, 
Edge ideals of oriented graphs, Internat. J. Algebra Comput. {\bf 29}
(2019), no. 3, 535--559.

\bibitem{Har}{F. Harary, {\it Graph Theory\/}, Addison-Wesley, 
Reading, MA, 1972.}

\bibitem{Radical-Herzog}{J. Herzog, 
Y. Takayama, N. Terai, 
\textit{On the radical of a monomial ideal}, 
Arch. Math. (Basel), {\bf 85}, (2005), 397--408.}



\bibitem{depth-monomial} J. Mart\'\i nez-Bernal, S. Morey, C. E.
Vivares and R. H. Villarreal, Depth and regularity 
of monomial ideals 
via polarization and combinatorial optimization,  
Acta Math. Vietnam. {\bf 44} (2019), no. 1, 243--268.


\bibitem{Mats}{H. Matsumura, {\it Commutative Ring Theory\/}, 
Cambridge
Studies in Advanced Mathematics {\bf 8}, 
Cambridge University Press, 1986.}

\bibitem{susan-reyes-vila} S. Morey, E. Reyes and R. H. Villarreal, 
Cohen--Macaulay, Shellable and unmixed clutters with a perfect
matching of K\"onig type, 
J. Pure Appl. Algebra {\bf 212} (2008), no. 7, 1770--1786.

\bibitem{PS}  C. Paulsen and S. Sather-Wagstaff, Edge ideals of
weighted graphs, J. Algebra Appl. {\bf 12} (2013), no. 5, 1250223. 
 

 \bibitem{WOG} Y. Pitones, E. Reyes and J. Toledo, 
Monomial ideals of weighted oriented graphs, Electron. J. Combin.
{\bf 26} (2019), no. 3, Paper 44, 18 pp.  
  
 \bibitem{disc-math} B. Randerath and P. D. Vestergaard, On
 well-covered graphs of odd girth $7$ or greater, Discuss. Math.
 Graph Theory {\bf 22} (2002), 159--172. 

\bibitem{ravindra} G. Ravindra, Well-covered graphs, 
J. Combinatorics Information Syst. Sci. {\bf 2} (1977), no. 1, 20--21. 

\bibitem{cm-graphs} R. H. Villarreal, Cohen--Macaulay graphs, Manuscripta
Math. {\bf 66} (1990), 277--293.

\bibitem{unmixed} R. H. Villarreal, Unmixed bipartite graphs, Rev. 
Colombiana Mat. {\bf 41} (2007), no. 2, 393--395. 

\bibitem{monalg-2} R. H. Villarreal, {\it Monomial Algebras, Second
Edition\/}, Monographs and Research Notes in Mathematics, Chapman and
Hall/CRC, 2015.  

\bibitem{Zhu-Xu} G. Zhu, L. Xu, H. Wang and Z. Tang, 
Projective dimensions and regularity of edge ideals of some weighted
oriented graphs, Rocky Mountain J. Math. {\bf 49} (2019), no. 4,
1391--1406.

\bibitem{Zhu-Xu-Wang-Zhang} G. Zhu, L. Xu, H. Wang and J. Zhang, 
Regularity and projective dimension of powers of edge ideal of the
disjoint union of some weighted oriented gap-free bipartite graphs.
Preprint, 2019, arxiv:1906.04682v1.
\end{thebibliography}

\end{document}